\documentclass[12pt]{amsart}
\usepackage{amsmath}
\usepackage{amssymb}
\usepackage{latexsym}

\newtheorem{theorem}{Theorem}[section]
\newtheorem{cor}[theorem]{Corollary}
\newtheorem{lemma}[theorem]{Lemma}

\newtheorem{prop}[theorem]{Proposition}

\newenvironment{proof*}{\vskip 2mm\noindent {}}{\hfill $\Box$ \vskip 2mm}
\numberwithin{equation}{section}
\newcommand{\C}{{\mathbb{C}}}
\newcommand{\D}{{\mathbb{D}}}
\newcommand{\G}{{\mathbb{G}}}

\newcommand{\E}{{\mathbb{E}}}

\begin{document}
\title{Geometric properties of the tetrablock}

\author{W\l odzimierz Zwonek}

\address{Institute of Mathematics, Faculty of Mathematics and Computer Science, Jagiellonian University,
\L ojasiewicza 6, 30-348 Krak\'ow, Poland}
\email{{Wlodzimierz.Zwonek}@im.uj.edu.pl}
\thanks{The work is partially supported by the grant of the Polish National Science Centre no.
UMO-2011/03/B/ST1/04758.}





\begin{abstract}
In this short note we show that the tetrablock is i $\C$-convex domain. In the proof of this fact a new class of ($\C$-convex) domains
is studied. The domains are natural caniddates to study on them the behavior of holomorphically invariant functions.
\end{abstract}

\maketitle

\section{Introduction}

Recently, two domains: symmetrized bidisc and tetrablock, arising from the $\mu$-synthesis, turned out to be
interesting examples in the geometric function theory. In particular, both domains are non-convex 
and even more, they cannot be exhausted by domains biholomorphic to convex ones and yet the Lempert function 
and the Carath\'eodory distance
coincide on them (see \cite{AY2004}, \cite{Cos}, \cite{Edi}, \cite{EKZ}). However, we have a more detailed knowledge on geometric properties
in the case of the symmetrized bidisc. In particular, it is known that
the symmetrized bidisc is $\C$-convex and may be exhausted by strongly linearly convex domains (see \cite{NPZ} 
and \cite{PZ}). All these facts show the importance of the domains from the point of view of the Lempert theorem on the equality of
the Lempert function and the Carath\'eodory distance (see
two papers of Lempert: \cite{Lem} and \cite{Lem1984}). We shall deal with analoguous properties of the tetrablock. 
More precisely, we show that
the tetrablock is $\C$-convex (see Corollary~\ref{cor:c-convexity}) which corrects the claim
stated in \cite{Kos} where, due to the typing error made 
in the formula describing the tetrablock, the converse was claimed.

It is interesting that the study of
geometric properties of the tetrablock can be reduced to considering domains being generalizations of the symmetrized bidisc which may lead
in the future to the study of other domains arising in the process
of symmetrization of a 'nice' pseudoconvex complete Reinhardt domain. The domains $\G_{2,\rho}$, which are $\C$-convex 
and aproximate the symmetrized bidisc, are natural candidates for the further study on the equality between the Lempert function and the Carath\'eodory distance
as well as on the possibility of exhausting them with strongly linearly convex domains.

Basic notions, definitions and properties from the theory of invariant functions, linearly and $\C$-convex domains that
we shall use in the paper may be found in \cite{JP}, \cite{APS} and \cite{Hor}. 

\section{Preliminary results}
Below we present analytic definitions of both domains that will be of interest to us.

Recall that  {\it tetrablock} may be defined as follows (see \cite{AWY})
\begin{equation}\label{def:tetrablock}
  \E=\{x\in\C^3:|x_1-\bar x_2x_3|+|x_2-\bar x_1x_3|+|x_3|^2<1\}
\end{equation}
and the {\it symmetrized bidisc} as follows (see e. g. \cite{AY2004})
\begin{equation}
  \G_2=\{(s,p)\in\C^2:|s-\bar sp|+|p|^2<1\}.
\end{equation}


Let us begin our study with a close relation between $\E$ and $\G_2$ which may be a good starting point for us.

\begin{lemma}[see \cite{Bha}]\label{lemma:1}
For any $x=(x_1,x_2,x_3)\in\C^3$ the following are equivalent
\begin{enumerate}
  \item  $x\in\E$,
  \item for any $\omega\in\C$ with $|\omega|=1$ we have $(x_1+\omega x_2,\omega x_3)\in\G_2$   
  \item for any $\omega\in\C$ with $|\omega|\leq 1$ we have $(x_1+\omega x_2,\omega x_3)\in\G_2$
\end{enumerate}
\end{lemma}
Actually, the above equivalence follows from the following observation
\begin{multline} 
  |x_1+\omega x_2-(\bar x_1+\bar\omega \bar x_2)\omega x_3|=|x_1-\bar x_2x_3+\omega(x_2-\bar x_1x_3)|\leq\\
|x_1-\bar x_2x_3|+|x_2-\bar x_1x_3|.
\end{multline}
together with the fact that the inequality above becomes equality for some $|\omega|=1$.

Let us also denote $\Phi_{\omega}(x):=(x_1+\omega x_2,\omega x_3)$, $x\in\mathbb E$, 
$\omega\in\overline{\mathbb D}$. Define $\sigma(x)=(x_2,x_1,x_3)$, $x\in\mathbb C^3$, 
$\Psi_{\omega}:=\Phi_{\omega}\circ\sigma$.

{\bf Remark.} The equality (\ref{def:tetrablock}) defining the tetrablock allows to find out that
any point $x\in\partial\mathbb E$ such that $x_1\neq \bar x_2x_3$ and $x_2\neq \bar x_1 x_3$ is a smooth boundary point of $\partial\E$.
Moreover, the condition $x\in\partial\E$ and $x_1=\bar x_2 x_3$ or $x_2=\bar x_1x_3$ means that
$x=(re^{i\theta},re^{i\tau},e^{i(\theta+\tau)})$ or $x=(e^{i\theta},re^{i\tau},re^{i(\theta+\tau)})$ or
$x=(re^{i\theta},e^{i\tau},re^{i(\theta+\tau)})$
for some $r\in[0,1]$, $\theta,\tau\in\mathbb R$. 
Composing with automorphism of $\E$ given by the formula 
$(e^{-i\theta}y_1,e^{-i\tau}y_2,e^{-i(\theta+\tau)}y_3)$
we shall often be able to reduce the problem to special cases: 
$x=(r,r,1)$ or $x=(1,r,r)$ or $x=(r,1,r)$, $r\in[0,1]$.

At first we shall prove that for any $x\not\in\E$ 
there is a hyperplane passing through $x$ and omitting $\E$. Thus we show the following.

\begin{lemma}\label{lemma:2}
 $\E$ is linearly convex.
\end{lemma}
\begin{proof}
Let $x\not\in\E$. Making use of the description of $\mathbb E$ from Lemma~\ref{lemma:1} 
for such an $x$ we find an $\omega\in\overline{\D}$ 
with $\Phi_{\omega}(x)=(x_1+\omega x_2,\omega x_3)\not\in\G_2$. 
The linear convexity of $\G_2$ (see e. g. \cite{NPZ}) implies
that there is a line $l=\{(s,p)\in\C^2:as+bp=c\}$ with  $\Phi_{\omega}(x)\in l$ and $l\cap\G_2=\emptyset$. 
Then $x\in L:=\{y\in\C^3:a y_1+a\omega y_2+b\omega y_3=c\}$ and as one may easily check $L\cap\E=\emptyset$. 
\end{proof}

As we shall see a more refined procedure than the one described above will lead us to the precise description of supporting hyperplanes
which will finally lead to the proof of $\C$-convexity of $\E$. But before that we need some notations and auxiliary results.

\section{Two-dimensional symmetrized domains characterizing the tetrablock}

Motivated by Lemma~\ref{lemma:1} we shall see that the symmetrized images of special domains of the form
\begin{equation}
 D_{\rho}:=\{z\in\mathbb C^2:|z_1|,|z_2|<1,|z_1z_2|<\rho\}
\end{equation}
where $\rho\in(0,1]$ will play a special role in the study of the geometry of $\mathbb E$.

Let us denote
\begin{equation}
\mathbb G_{2,\rho}:=\pi(D_{\rho}),\;\rho\in(0,1]
\end{equation}
where $\pi(z):=(z_1+z_2,z_1z_2)$, $z\in\C^2$. Recall that $\G_2=\G_{2,1}$.


Then it follows from Lemma~\ref{lemma:1} that $\Phi_{\omega}(\mathbb E)\subset\mathbb G_{2,|\omega|}$, $0<|\omega|\leq 1$. 
We shall see that we even have the equality.
\begin{prop}\label{prop:sur}
\begin{equation}
 \Phi_{\omega}(\mathbb E)=\mathbb G_{2,|\omega|},\; 
\omega\in\overline{\mathbb D}\setminus\{0\}.
\end{equation}
\end{prop}
\begin{proof}
Actually, let $\rho:=|\omega|$ 
and take $(s,p)\in\mathbb G_{2,\rho}$. Put 
\begin{equation}
 x:=\left(\frac{s-\overline{s}p}{1-|p|^2},\frac{\overline{s}-s\overline{p}}{1-|p|^2}\frac{p}{\omega},\frac{p}{\omega}\right).
\end{equation}
It easily follows from (\ref{def:tetrablock}) that $x\in\mathbb E$ and $\Phi_{\omega}(x)=(s,p)$.
\end{proof}

A domain $D\subset\C^n$ is called {\it $\C$-convex} if for any affine complex line $l$ such that $l\cap D\neq\emptyset$ the set $l\cap D$ 
is connected and simply connected.
 
For a domain $D\subset\mathbb C^n$ and a point $a\in\C^n$ we denote by
$\Gamma_D(a)$ the set of all complex hyperplanes $L$ such that $(a+L)\cap D=\emptyset$. 
We shall often understand this set as the subset of $\mathbb P^{n-1}$: $L=\{x\in\C^n:\langle x,b\rangle=0\}$ 
is identified with $[b]\in\mathbb P^{n-1}$. 

Recall the basic criterion on $\C$-convexity that we shall use: the bounded domain $D\subset\C^n$, $n>1$, is $\C$-convex iff for an $x\in\partial D$ the set $\Gamma_D(x)$ is 
non-empty and connected (cf. e. g. Theorem 2.5.2 in \cite{APS}).

{\bf Remark.} It is elementary to see that for $n\geq 2$ 
\begin{equation}
\Gamma_{\mathbb D^n}(1,\ldots,1)=\{[(t_1,\ldots,t_n)]:(t_1,\ldots,t_n)\in[0,\infty)^n\setminus\{0\}\}.
\end{equation}

{\bf Remark.} The boundary point $(s,p)=\pi(\lambda_1,\lambda_2)$ 
of the domain $\mathbb G_{2,\rho}$, $\rho\in(0,1)$ is not smooth iff $\{|\lambda_1|,|\lambda_2|\}=\{1,\rho\}$. 

One may also easily see that 
\begin{itemize}
\item if $|\lambda_1|\leq \rho$, $|\lambda_2|=1$ 
then $\Gamma_{\mathbb G_{2,\rho}}(\pi(\lambda_1,\lambda_2))\supset\{[(-\lambda_2,1)]\}$ 
(if $|\lambda_1|<\rho$ then the inclusion becomes the equality),
\item if $|\lambda_1||\lambda_2|=\rho$, $\rho\leq |\lambda_j|\leq 1$, $j=1,2$, 
then $\Gamma_{\mathbb G_{2,\rho}}(\pi(\lambda_1,\lambda_2))\supset
\{[(0,1)]\}$ (if  $\rho<|\lambda_j|<1$, $j=1,2$ then the inclusion becomes the equality).
\end{itemize}

Let us formulate a result which essentially reduces the problem of description of $\Gamma_{\E}$ to that
of $\Gamma_{\G_{2,\rho}}$ (and $\Gamma_{\D^2}$).

\begin{lemma}\label{lemma:3}
 Let $x\in\C^3$ and let $0<|\omega|\leq 1$. Then the following are equivalent
\begin{itemize}
\item $[(a,c)]\in\Gamma_{\G_{2,|\omega|}}(\Phi_{\omega}(x))$ (respectively,  $[(a,c)]\in\Gamma_{\G_{2,|\omega|}}(\Psi_{\omega}(x))$),
\item $[(a,\omega a,\omega c)]\in \Gamma_{\E}(x)$ (respectively, 
$[(\omega a,a,\omega c)]\in\Gamma_{\E}(x)$).
\end{itemize}
\end{lemma}
\begin{proof} 

Let $l$ such that $l-\Phi_{\omega}(x)\in\Gamma_{\mathbb G_{2,|\omega|}}\Phi_{\omega}(x)$ be given by the equality
$as+cp=d$ (i. e. $[(a,b)]\in\Gamma_{\mathbb G_{2,|\omega|}}(\Phi_{\omega}(x))$). 
Then the equality $ay_1+\omega ay_2+\omega cy_3=d$ defines a hyperplane omitting
$\E$ and thus $[(a,\omega a,\omega c)]\in\Gamma_{\E}(\Phi_{\omega}(x))$.

To show the other implication let $L$ be such that $L-x\in\Gamma_{\mathbb E}(x)$. Let $L$ be given by the equation
$ay_1+\omega ay_2+\omega cy_3=d$ which may be written as
$a(y_1+\omega y_2)+c\omega y_3=d$. Then
from the equality $\Phi_{\omega}(\mathbb E)=\mathbb G_{2,|\omega|}$ (Lemma~\ref{lemma:2}) we get that the line given by the equality
$a s+cp=0$ belongs to $\Gamma_{\mathbb G_{2,|\omega|}}(\Phi_{\omega}(x))$.

\end{proof}

\begin{theorem}\label{thm:1}
Let $r\in[0,1]$. Then
\begin{multline}\label{case1}
\Gamma_{\mathbb E}(r,r,1)=\\
\bigcup\sb{0<|\omega|\leq 1}
\left\{[(\tilde s,\omega\tilde s,\omega\tilde p)],[(\omega\tilde s,\tilde s,\omega\tilde p)]:
[(\tilde s,\tilde p)]\in\Gamma_{\mathbb G_{2,|\omega|}}(r+r\omega,\omega)\right\}\cup\\
\{[(\tilde s,0,\tilde p)],[(0,\tilde s,\tilde p)]:
[(\tilde s,\tilde p)]\in\Gamma_{\mathbb D^2}(r,1)\}.
\end{multline}
Let $r\in[0,1)$. Then
\begin{equation}\label{case2}
 \Gamma_{\mathbb E}(1,r,r)=\{[(-1,-\omega,\omega):\omega\in\overline{\mathbb D}\}.
\end{equation}

\end{theorem}
{\bf Remark.} Smoothness of the boundary point $x$ together with the linear convexity means that $\Gamma_{\E}(x)$ is a singleton.
Note also that it follows from the earlier remark that the cases of boundary points considered in Theorem~\ref{thm:1} 
(i. e. $(r,r,1)$ and $(1,r,r)$)
represent all (up to linear automorphisms of $\E$) non-smooth boundary points and it means that Theorem~\ref{thm:1}
gives a complete description of $\Gamma_{\mathbb E}$.

{\bf Remark.} Note that $\Psi_{\omega}(1,r,r)=(r+\omega,r\omega)\in\mathbb G_{2,|\omega|}$, $r\in[0,1)$ and $\Phi_{\omega}(1,r,r)=
(1+r\omega,r\omega)\in\partial \mathbb G_{2,|\omega|}$ which implies that
the equality in (\ref{case2}) may be expressed in the form which would also be applicable in (\ref{case1}). 
We may namely write the right side
of (\ref{case2}) as follows:
\begin{multline}
 \bigcup\sb{0<|\omega|\leq 1}
\left\{[(\tilde s,\omega\tilde s,\omega\tilde p)]:[(\tilde s,\tilde p)]
\in\Gamma_{\mathbb G_{2,|\omega|}}(\Phi_{\omega}(1,r,r))\right\}\cup\\
\bigcup\sb{0<|\omega|\leq 1}\left\{[(\omega\tilde s,\tilde s,\omega\tilde p)]:
[(\tilde s,\tilde p)]\in\Gamma_{\mathbb G_{2,|\omega|}}(\Psi_{\omega}(1,r,r))\right\}\cup\\
\left\{[(\tilde s,0,\tilde p)]:[(\tilde s,\tilde p)]\in\Gamma_{\D^2}(1,r)\right\}.
\end{multline}

\begin{proof}
Note that in the case of the point $(r,r,1)$ we get that $\Phi_{\omega}(r,r,1)=\Psi_{\omega}(r,r,1)=(r+\omega r,\omega)\in\partial\E$, 
$\omega\in\overline{D}$. In this case the incusion '$\supset$' follows from the first part of Lemma~\ref{lemma:3}
and the fact that $\E\subset \D^3$.

To show the other inclusion let $L$ be such that $L-(r,r,1)\in\Gamma_{\mathbb E}(r,r,1)$. Let $L$ be given by the equation
$ay_1+by_2+cy_3=d$. Let $b=\omega a$ where $|\omega|\leq 1$ (the other case will be dealt with analoguously). If $\omega\neq 0$ then 
in view of the second part of Lemma~\ref{lemma:3} we get that $[(b,\frac{c}{\omega})]
\in\Gamma_{\mathbb G_{2,|\omega|}}(r+r\omega,\omega)$. 

Consider now $\omega=0$.
Then from the fact that $(r,r,1)+\lambda(-c,0,a)+\mu(0,1,0)\not\in\mathbb E$ for any $\lambda,\mu\in\mathbb C$ we get that the point 
$(r-c\lambda,1+a\lambda)\not\in\mathbb D^2$ - the last property follows from the fact that if the point 
$(y_1,y_3)$ satisfies the inequality
$|y_1-\bar\lambda y_3|+|\lambda-\bar y_1y_3|\geq 1-|y_3|^2$ for any $\lambda\in\mathbb C$ then $(y_1,y_3)\not\in\mathbb D^2$ 
(take e. g. $\lambda=\bar y_1y_3$). This finishes the proof of the case (\ref{case1}).

Let us consider now the point $(1,r,r)$, $r\in[0,1)$. Then  $\Phi_{\omega}(1,r,r)=(1+r\omega,r\omega)\in\partial \G_{2,|\omega|}$ 
and $\Psi_{\omega}(1,r,r)=(\omega+r,\omega r)\in\G_{2,|\omega|}(\omega+r,r\omega)$, $0<|\omega|\leq 1$. 

Note that $[(-1,1)]\in\Gamma_{\G_{2,|\omega|}}(1+r\omega,r\omega)$, $0<|\omega|\leq 1$ which in view of Lemma~\ref{lemma:3} implies 
that $[(-1,-\omega,\omega)]\in\Gamma_{\E}(1,r,r)$, $0<|\omega|\leq 1$. Certainly $[(-1,0,0)]\in\Gamma_{\E}(1,r,r)$. And this gives the 
inclusion $\supset$' in the case (\ref{case2}).

To show the other inclusion we proceed similarly as in the first case. Let $L$ be such that $L-(1,r,r)\in\Gamma_{\mathbb E}(1,r,r)$. 
Let $L$ be given by the equation
$ay_1+by_2+cy_3=d$. First note that $b=\omega a$ for some $|\omega|\leq 1$. Actually in the other case $a,b\neq 0$ and $a=\omega b$ 
for some $0<|\omega|<1$ 
and then in view of Lemma~\ref{lemma:3} we get that $[(a,\frac{c}{\omega})]\in\Gamma_{\G_{2,|\omega|}}(r+\omega,r\omega)=\emptyset$ 
- contradiction. 

If $\omega\neq 0$ then 
in view of Lemma~\ref{lemma:3} we get that $[(a,\frac{c}{\omega})]
\in\Gamma_{\mathbb G_{2,|\omega|}}(1+r\omega,r\omega)=\{[(-1,1)]\}$ so $[(a,b,c)]=[(-1,-\omega,\omega)]$. 

Consider now $\omega=0$.
Then as in the first case $[(a,c)]\in\Gamma_{\D^2}(1,r)$ which equals $\{[(1,0)]\}$ which finishes the proof.


\end{proof}

\section{$\C$-convexity of $\G_{2,\rho}$ and the tetrablock.}

In view of the above results we see that crucial for the proof of $\mathbb C$-convexity of the tetrablock is to find
the description of $\Gamma_{\mathbb G_{2,\rho}}$. Recall that $\G_2=\G_{2,1}$ is $\C$-convex (see \cite{NPZ}). We have the following.

\begin{theorem}\label{thm:2}
 $\mathbb G_{2,\rho}$ is $\mathbb C$-convex for any $\rho\in(0,1]$. 

Moreover, if $(s,p)=\pi_2(\lambda_1,\lambda_2)\in\partial \mathbb G_{2,\rho}$
then for any $\rho\in(0,1)$ we have
\begin{itemize}
 \item $\Gamma_{\mathbb G_{2,\rho}}(s,p)=\{[(-\lambda_2,1)]\}$ if $|\lambda_1|<\rho$, $|\lambda_2|=1$,
  \item $\Gamma_{\mathbb G_{2,\rho}}(s,p)=\{[(0,1)]\}$, if $|\lambda_1\lambda_2|=\rho$ and $|\lambda_1|,|\lambda_2|\in(\rho,1)$,
  \item the set $\{\frac{\tilde s}{\tilde p}:[(\tilde s,\tilde p)]\in\Gamma_{\mathbb G_{2,\rho}}(s,p)\}$ contains $0$ and is a convex set 
if $\{|\lambda_1|,|\lambda_2|\}=\{1,\rho\}$.
\end{itemize}
\end{theorem}
\begin{proof} Let $\rho\in(0,1)$. The equalities in the first two cases follow from earlier remarks and the fact that the points considered there 
are smooth.


Let us consider the third case. Note that any complex line passing through $(s,p)$ and through a point from $\mathbb G_{2,\rho}$ 
must contain a point from the set $\pi(D_{\rho}\cap\{(\mu_1,\mu):\rho<|\mu_1|\})$ which implies that 
the set of points not in the set considered are the points of the form
\begin{equation}
 \frac{\lambda_1+\lambda_2-\mu_1-\mu}{\lambda_1\lambda_2-\mu_1\mu}
\end{equation}
where $\rho<|\mu_1|<1$, $|\mu\mu_1|<\rho$. The last may be given in the form
\begin{equation}
 \frac{1}{\lambda_1\lambda_2}\left(\frac{\lambda_1\lambda_2}{\mu_1}+\frac{\lambda_1+\lambda_2-\mu_1-
\frac{\lambda_1\lambda_2}{\mu_1}}{1-\frac{\mu_1\mu}{\lambda_1\lambda_2}}\right).
\end{equation}
Since the function $z\to 1/(1-z)$ maps the unit disc to $\{\operatorname{Re}z>1/2\}$ we easily get that for the fixed 
$\mu_1$ the set of points of the previous form for all $\mu$ with $|\mu_1||\mu|<\rho$ is an open half plane. 
Now the set of numbers of the set in the theorem
is the intersection of the complements of the sets of the last form. This implies that it is convex. We already know that 
$[(0,1)]\in\Gamma_{\G_{2,\rho}}(s,p)$ which finishes the proof. The fact that all sets $\Gamma_{\G_{2,\rho}}(x)$, $x\in\partial\G_{2,\rho}$,
are non-empty and connected implies the $\C$-convexity of $\G_{2,\rho}$. 

\end{proof}

\begin{cor}\label{cor:c-convexity}
 $\mathbb E$ is $\mathbb C$-convex.
\end{cor}
\begin{proof} Linear convexity of $\E$ implies that in the case of a smooth boundary point $x\in\partial\E$
the set $\Gamma_{\E}(x)$ is a singleton. Consider then the non-smooth point $x\in\partial\E$. It is sufficient to consider the cases
\begin{itemize}
 \item $x=(r,r,1)$, 
$r\in[0,1]$,
\item $x=(1,r,r)$, $r\in[0,1)$.
\end{itemize}
Theorem~\ref{thm:1} together with Theorem~\ref{thm:2} (we also need to know that $\Gamma_{\G_2}(\pi(\lambda))$
is connected and contains the point $[(0,1)]$ - this follows from \cite{NPZ}) 
imply that the set $\Gamma_{\mathbb E}(r,r,1)$ is the union of connected sets 
whose intersection is non-empty (it contains the point $[(0,0,1)]$) so it is connected too. 

The fact that $\Gamma_{\mathbb E}(1,r,r)$ is connected follows immediately from its description. This finishes the proof.
\end{proof}


\begin{thebibliography}{}


\bibitem{AWY} A.~A.~Abouhajar, M.~C.~White \& N.~J.~Young,
    {\it A Schwarz lemma for a domain related to
    mu-synthesis}, J. Geom. Analysis, 17(4) (2007), 717-750.


\bibitem{AY2004}
J.~Agler \& N.~J.~Young, \textit{The hyperbolic geometry of the
symmetrized bidisc},  J. Geom. Anal.  14  (2004),  no. 3,
375--403.

\bibitem{APS} M. Andersson, M. Passare, R. Sigurdsson, {\it Complex convexity and
analytic functionals}, Birkh\"auser, Basel-Boston-Berlin, 2004.

\bibitem{Bha} T. Bhattacharyya, {\it Operator theory on the tetrablock}, preprint, 2012,
http://arxiv.org/abs/1207.3395.

\bibitem{Cos} C. ~Costara, \textit{The symmetrized bidisc and Lempert's theorem}, 
Bull. London Math. Soc. 36 (2004), no. 5, 656--662.

\bibitem{Edi} A.~Edigarian, \textit{ A note on C. Costara's paper: 
``The symmetrized bidisc and Lempert's theorem'' [Bull. London Math. Soc. 36 (2004), no. 5, 656--662]}, 
Ann. Polon. Math. 83 (2004), no. 2, 189--191.

\bibitem{EKZ} A. Edigarian, \L. Kosi\'nski, W. Zwonek, {\it The Lempert Theorem and the tetrablock}, J. Geom. Analysis,
to appear.

\bibitem{Hor} L. H\"ormander, {\it Notions of Convexity},
Birkh\"auser, Basel-Boston-Berlin, 1994.


\bibitem{JP} M. Jarnicki, P. Pflug, {\it Invariant distances and metrics in
complex analysis}, de Gruyter, Berlin-New York, 1993.

\bibitem{Kos} \L. Kosi\'nski, {\it Geometry of quasi-circular domains and applications to tetrablock},
Proc. Amer. Math Soc., 139 (2011), 559-569.

  
\bibitem{Lem} L.~Lempert, {\it La m\'etrique de Kobayashi et la repr\'esentation des
domaines sur la boule}, Bull. Soc. Math. France 109 (1981),
427--474.

\bibitem{Lem1984} L.~ Lempert, {\it Intrinsic distances and holomorphic retracts}, in Complex analysis and applications ’81
(Varna, 1981 ), 341-364, Publ. House Bulgar. Acad. Sci., Soﬁa, 1984

\bibitem{NPZ} N. Nikolov, P. Pflug, W. Zwonek,
{\it An example of a bounded $\C$-convex domain which is not
biholomorphic to a convex domain}, Math. Scand. 102 (2008),
149-155.

\bibitem{PZ} P. Pflug, W. Zwonek, {\it Exhausting domains of the symmetrized bidisc}, 
Ark. Mat. 50 (2012), no. 2, 397-402.

\end{thebibliography}
\end{document}